\documentclass[12pt]{article}
\usepackage[latin1]{inputenc}
\usepackage{cmap}
\usepackage{lmodern}
\usepackage[british]{babel}

\usepackage{amssymb, amsmath, amsthm}
\usepackage[a4paper,top=25mm,bottom=25mm,left=25mm,right=25mm]{geometry}
\usepackage{etex}
\usepackage{ragged2e}

\usepackage{authblk} 
\usepackage{pifont}
\usepackage{graphicx}
\usepackage[usenames,dvipsnames,svgnames,table]{xcolor}
\usepackage[figuresright]{rotating}
\usepackage{xtab} 
\usepackage{longtable} 
\usepackage{multirow}
\usepackage{footnote}
\usepackage[stable]{footmisc}
\usepackage{chngpage} 
\usepackage{pdflscape} 
\usepackage{tocbibind} 

\usepackage{pgfplots}
\pgfplotsset{compat=1.14}
\usepackage{setspace}

\usepackage{array}
\newcolumntype{K}[1]{>{\centering\arraybackslash$}p{#1}<{$}}

\makesavenoteenv{tabular}
\usepackage{tabularx}
\usepackage{booktabs}
\usepackage{threeparttable}
\usepackage[referable]{threeparttablex} 
\newcolumntype{R}{>{\raggedleft\arraybackslash}X}
\newcolumntype{L}{>{\raggedright\arraybackslash}X}
\newcolumntype{C}{>{\centering\arraybackslash}X}
\newcolumntype{A}{>{\columncolor{gray!25}}C}
\newcolumntype{a}{>{\columncolor{gray!25}}c}

\newlength{\tablen}

\usepackage{dcolumn} 
\newcolumntype{.}{D{.}{.}{-1}}

\usepackage{tikz}
\usetikzlibrary{arrows}
\usepackage[semicolon]{natbib}
\usepackage[hyphens]{url}
\usepackage{hyperref} 
\hypersetup{
  colorlinks   = true,    
  urlcolor     = blue,    
  linkcolor    = blue,    
  citecolor    = red      
}
\usepackage{microtype}
\usepackage[justification=centerfirst]{caption} 

\usepackage[labelformat=simple]{subcaption}

\DeclareCaptionLabelFormat{parenthesis}{(#2)}
\makeatletter
\renewcommand\p@subfigure{\arabic{figure}.}
\makeatother

\DeclareCaptionLabelFormat{parenthesis}{(#2)}
\makeatletter
\renewcommand\p@subtable{\arabic{table}.}
\makeatother

\usepackage{enumitem}

\setlist[itemize]{leftmargin=2.5\parindent}
\setlist[enumerate]{leftmargin=2.5\parindent}

\theoremstyle{plain}

\newtheorem{corollary}{Corollary}[section]

\newtheorem{proposition}{Proposition}[section]
\newtheorem{theorem}{Theorem}[section]

\theoremstyle{definition}
\newtheorem{axiom}{Axiom}[section]

\newtheorem{definition}{Definition}[section]
\newtheorem{example}{Example}[section]

\theoremstyle{remark}

\def\keywords{\vspace{.5em} 
{\noindent \textit{Keywords}:\,}}

\def\JEL{\vspace{.5em} 
{\noindent \textbf{\emph{JEL} classification number}:\,}}

\def\AMS{\vspace{.5em} 
{\noindent \textbf{\emph{MSC} class}:\,}}

\author{\href{https://sites.google.com/site/laszlocsato87}{L\'aszl\'o Csat\'o}\thanks{~e-mail: laszlo.csato@uni-corvinus.hu} }
\affil{Institute for Computer Science and Control, Hungarian Academy of Sciences (MTA SZTAKI) \\
Laboratory on Engineering and Management Intelligence, Research Group of Operations Research and Decision Systems}
\affil{Corvinus University of Budapest (BCE) \\
Department of Operations Research and Actuarial Sciences}
\affil{Budapest, Hungary}
\title{A characterization of the Logarithmic Least Squares Method}

\begin{document}

\maketitle

\begin{abstract}
We provide an axiomatic characterization of the Logarithmic Least Squares Method (sometimes called row geometric mean), used for deriving a preference vector from a pairwise comparison matrix.
This procedure is shown to be the only one satisfying two properties, correctness in the consistent case, which requires the reproduction of the inducing vector for any consistent matrix, and invariance to a specific transformation on a triad, that is, the weight vector is not influenced by an arbitrary multiplication of matrix elements along a 3-cycle by a positive scalar.


\JEL{C44}

\AMS{90B50, 91B08}
\end{abstract}

\section{Introduction}

Pairwise comparisons are a fundamental tool in many decision-analysis methods such as the Analytic Hierarchy Process (AHP) \citep{Saaty1980}.
However, in real-world applications the judgements of decision-makers may be inconsistent: for example, alternative $A$ is two times better than alternative $B$, alternative $B$ is three times better than alternative $C$, but alternative $A$ is not six times better than alternative $C$. Inconsistency can also be an inherent feature of the data, for example, on the field of sports \citep{Csato2013a, BozokiCsatoTemesi2016, ChaoKouLiPeng2018}.

Therefore, a lot of methods have been suggested in the literature for deriving preference values from pairwise comparison matrices.
In such cases, it seems to be fruitful to follow an axiomatic approach: the introduction and justification of reasonable properties may help to narrow the set of appropriate weighting methods and reveal some crucial features of them.
The most important contribution of similar analyses can be an \emph{axiomatic characterization}, when a set of properties \emph{uniquely} determine a preference vector.

Characterization of different methods is a standard tool in social choice theory, for instance, in the case of the Shapley value in game theory (see e.g. \citet{Shapley1953, vandenBrinkPinter2015}), or for the Hirsch index in scientometrics (see e.g. \citet{Woeginger2008, BouyssouMarchand2014}).
This approach has been applied recently for inconsistency indices of pairwise comparison matrices \citep{Csato2018e, Csato2018a}.

\citet{Fichtner1984}, presumably the first work on the axiomatizations of weighting methods,  characterized the Logarithmic Least Squares Method \citep{Rabinowitz1976, CrawfordWilliams1980, CrawfordWilliams1985, DeGraan1980} by using four requirements, correctness in the consistent case, comparison order invariance, smoothness, and power invariance.
\citet{Fichtner1986} showed that substituting power invariance with rank preservation leads to the Eigenvector Method suggested by \citet{Saaty1980}.

From this set of axioms, correctness in the consistent case and comparison order invariance are almost impossible to debate. However, according to \citet{Bryson1995}, there exists a goal-programming method satisfying power invariance and a slightly modified version of smoothness, which possesses the additional property that the presence of a single outlier cannot prevent the identification of the correct priority vector. While \citet{Fichtner1984} introduces smoothness in terms of differentiable functions and continuous derivatives, the interpretation of \citet{Bryson1995} -- a small change in the input does not lead to a large change in the output -- seems to be more natural for us. \citet{CookKress1988} approached the problem by focusing on distance measures in order to get another goal programming method on an axiomatic basis.

Smoothness and power invariance can be entirely left out from the characterization of the Logarithmic Least Squares Method. \citet{BarzilaiCookGolany1987} exchange them for a consistency-like axiom by considering two procedures: (1) some pairwise comparison matrices are aggregated to one matrix and the solution is computed for this matrix; (2) the priorities are derived separately for each matrix and combined by the geometric mean; which are required to result in the same preference vector. We think it is not a simple condition immediately to adopt.
\citet{Barzilai1997} managed to replace this axiom and comparison order invariance with essentially demanding that each individual weight is a function of the entries in the corresponding row of the pairwise comparison matrix only. Joining to \citet{Dijkstra2013}, we are also somewhat uncomfortable with this premise.

To summarize, the problem of weight derivation does not seem to be finally settled by the axiomatic approach. Consequently, it may not be futile to provide another characterization of the Logarithmic Least Squares Method, which hopefully highlights some new aspects of the procedure.
This is the main aim of the current paper.

Presenting an axiomatic characterization does not mean that we accept all properties involved as wholly justified and unquestionable or we reject the axioms proposed by previous works.
To consider an example from a related topic, although most axiomatic analysis of inconsistency \citep{BrunelliFedrizzi2011, Brunelli2016a, Brunelli2017, BrunelliFedrizzi2015, BrunelliFedrizzi2018, CavalloDApuzzo2012, KoczkodajSzwarc2014, KoczkodajSzybowski2015, KoczkodajUrban2018} look for well-motivated axioms that should be satisfied by any reasonable measure, \citet{Csato2018a} does not deal with the appropriate motivation of his axioms, the issue to be investigated is only how they can narrow the set of inconsistency indices.

This paper strictly follows the latter direction, therefore, we only say that \emph{if} one agrees with our axioms, then the geometric mean remains the only choice.
The importance of an axiomatic characterization does not depend on a convincing explanation for the desirability of the properties suggested.
For example, Arrow's impossibility theorem \citep{Arrow1950} can also be interpreted as an axiomatization of the dictatorial rule by universality, Pareto efficiency, and independence of irrelevant alternatives, however, it does not mean that the dictatorial rule is good.


The study is structured as follows. Section~\ref{Sec2} presents some definitions on the field of pairwise comparison matrices. Two properties of weighting methods are defined in Section~\ref{Sec3}, which will provide the characterization of the Logarithmic Least Squares Method in Section~\ref{Sec4}. Section~\ref{Sec5} summarizes our findings.

\section{Preliminaries} \label{Sec2}

Assume that $n$ alternatives should be measured with respect to a given criterion on the basis of pairwise comparisons such that $a_{i,j}$ is an assessment of the relative importance of alternative $i$ with respect to alternative $j$.

Let $\mathbb{R}^{n}_+$ and $\mathbb{R}^{n \times n}_+$ denote the set of positive (with all elements greater than zero) vectors of size $n$ and matrices of size $n \times n$, respectively.

\begin{definition} \label{Def21}
\emph{Pairwise comparison matrix}:
Matrix $\mathbf{A} = \left[ a_{i,j} \right] \in \mathbb{R}^{n \times n}_+$ is a \emph{pairwise comparison matrix} if $a_{j,i} = 1/a_{i,j}$ for all $1 \leq i,j \leq n$.
\end{definition}

Any pairwise comparison matrix is well-defined by its elements above the diagonal since we discuss only multiplicative pairwise comparison matrices with the reciprocal property throughout the paper.
Let $\mathcal{A}^{n \times n}$ be the set of pairwise comparison matrices of size $n \times n$.


A pairwise comparison matrix $\mathbf{A} \in \mathcal{A}^{n \times n}$ is called \emph{consistent} if $a_{i,k} = a_{i,j} a_{j,k}$ for all $1 \leq i,j,k \leq n$. Otherwise, it is said to be \emph{inconsistent}. Any pairwise comparison matrix is allowed to be inconsistent unless its consistency is explicitly stated.

\begin{definition} \label{Def22}
\emph{Weight vector}:
Vector $\mathbf{w}  = \left[ w_{i} \right] \in \mathbb{R}^n_+$ is a \emph{weight vector} if $\sum_{i=1}^n w_{i} = 1$.
\end{definition}

Let $\mathcal{R}^{n}$ be the set of weight vectors of size $n$.

\begin{definition} \label{Def23}
\emph{Weighting method}:
Function $f: \mathcal{A}^{n \times n} \to \mathcal{R}^{n}$ is a \emph{weighting method}.
\end{definition}

A weighting method associates a weight vector to any pairwise comparison matrix $\mathbf{A}$ such that $f_i(\mathbf{A})$ is the weight of alternative $i$.

Several weighting methods have been suggested in the literature, see \citet{ChooWedley2004} for an overview. This paper discusses two of them, which are among the most popular.

\begin{definition} \label{Def24}
\emph{Eigenvector Method} ($EM$) \citep{Saaty1980}:
The \emph{Eigenvector Method} is the function $\mathbf{A} \to \mathbf{w}^{EM} (\mathbf{A})$ such that
\[
\mathbf{A} \mathbf{w}^{EM}(\mathbf{A}) = \lambda_{\max} \mathbf{w}^{EM}(\mathbf{A}) \qquad \text{and} \qquad \sum_{i=1}^n w_i^{EM} = 1,
\]
where $\lambda_{\max}$ denotes the maximal eigenvalue, also known as principal or Perron eigenvalue, of matrix $\mathbf{A}$.
\end{definition}

\begin{definition} \label{Def25}
\emph{Logarithmic Least Squares Method} ($LLSM$) \citep{CrawfordWilliams1980, CrawfordWilliams1985, DeGraan1980}: 
The \emph{Logarithmic Least Squares Method} is the function $\mathbf{A} \to \mathbf{w}^{LLSM} (\mathbf{A})$ such that the weight vector $\mathbf{w}^{LLSM} (\mathbf{A})$ is the optimal solution of the problem:
\begin{equation} \label{Eq_LLSM}
\min_{\mathbf{w} \in \mathcal{R}^n} \sum_{i=1}^n \sum_{j=1}^n \left[ \log a_{i,j} - \log \left( \frac{w_i}{w_j} \right) \right]^2.
\end{equation}
\end{definition}

$LLSM$ is sometimes called (row) \emph{geometric mean} because the solution of \eqref{Eq_LLSM} can be computed as
\begin{equation} \label{Form_LLSM}
w_i^{LLSM}(\mathbf{A}) = \frac{\prod_{j=1}^n a_{i,j}^{1/n}}{\sum_{k=1}^n \prod_{j=1}^n a_{k,j}^{1/n}}.
\end{equation}


\section{Axioms} \label{Sec3}

In this section, two properties of weighting methods will be discussed.

\begin{axiom} \label{Axiom1}
\emph{Correctness} ($CO$):
Let $\mathbf{A} \in \mathcal{A}^{n \times n}$ be a consistent pairwise comparison matrix.
Weighting method $f: \mathcal{A}^{n \times n} \to \mathcal{R}^n$ is \emph{correct} if $f_i(\mathbf{A}) / f_j(\mathbf{A}) = a_{i,j}$ for all $1 \leq i,j \leq n$.
\end{axiom}

$CO$ requires the reproduction of the inducing vector for any consistent pairwise comparison matrix. It was introduced by \citet{Fichtner1984} under the name \emph{correct result in the consistent case} and was used by \citet{Fichtner1986}, \citet{BarzilaiCookGolany1987} and \citet{Barzilai1997}, among others.

\begin{proposition} \label{Prop31}
The Eigenvector Method and the Logarithmic Least Squares Method satisfy correctness.
\end{proposition}

\begin{definition} \label{Def31}
\emph{$\alpha$-transformation on a triad}:
Let $\mathbf{A} \in \mathcal{A}^{n \times n}$ be a pairwise comparison matrix and $1 \leq i,j,k \leq n$ be three different alternatives. An \emph{$\alpha$-transformation on the triad} $(i,j,k)$ -- which is determined by the three alternatives $i$, $j$, and $k$ -- provides the pairwise comparison matrix $\hat{\mathbf{A}} \in \mathcal{A}^{n \times n}$ such that $\alpha > 0$, $\hat{a}_{i,j} = \alpha a_{i,j}$ ($\hat{a}_{j,i} = a_{j,i} / \alpha$), $\hat{a}_{j,k} = \alpha a_{j,k}$ ($\hat{a}_{k,j} = a_{k,j} / \alpha$), $\hat{a}_{k,i} = \alpha a_{k,i}$ ($\hat{a}_{i,k} = a_{i,k} / \alpha$) and $\hat{a}_{\ell, m} = a_{\ell, m}$ for all other elements.
\end{definition}

The transformation changes three elements of a pairwise comparison matrix along a 3-cycle. It can reproduce \emph{local} consistency: the choice $\alpha = \sqrt[3]{a_{i,k} / (a_{i,j} a_{j,k})}$ leads to $\hat{a}_{i,j} \hat{a}_{j,k} = \alpha^2 a_{i,j} a_{j,k} = a_{i,k} / \alpha = \hat{a}_{i,k}$.

Naturally, this process modifies all values of the triad, while maybe two of the comparisons are accurate and one contains all the inaccuracy. However, if no further information is available, then our assumption seems to be reasonable.

\begin{axiom} \label{Axiom2}
\emph{Invariance to $\alpha$-transformation on a triad} ($IT$):
Let $\mathbf{A}, \hat{\mathbf{A}} \in \mathcal{A}^{n \times n}$ be any two pairwise comparison matrices such that $\hat{\mathbf{A}}$ can be obtained from $\mathbf{A}$ through an $\alpha$-transformation on a triad.
Weighting method $f: \mathcal{A}^{n \times n} \to \mathcal{R}^n$ is \emph{invariant to $\alpha$-transformation on a triad} if $f(\mathbf{A}) = f ( \hat{\mathbf{A}} )$.
\end{axiom}

$IT$ means that the weights of the alternatives are not influenced by $\alpha$-transformations on triads. It has been inspired by the axiom \emph{independence of circuits} in \citet{Bouyssou1992}.

A motivation for invariance to $\alpha$-transformation on a triad can be the following. Consider a sports competition where player $i$ has defeated player $j$, player $j$ has defeated player $k$, while player $k$ has defeated player $i$, and suppose that the three wins are equivalent. Then the final ranking is not allowed to change if the margins of victories are modified by the same amount. In particular, the three results can be reversed ($j$ beats $i$, $k$ beats $j$, and $i$ beats $k$), or all comparisons can become a draw.
$IT$ is practically a generalization of this idea.

\begin{proposition} \label{Prop32}
The Eigenvector Method violates invariance to $\alpha$-transformation on a triad.
\end{proposition}

\begin{proof}
Consider the following pairwise comparison matrices:
\[
\mathbf{A} = \left[
\begin{array}{K{1.5em} K{1.5em} K{1.5em} K{1.5em}}
    1     & 1     & 1     & 8 \\
    1     & 1     & 1     & 1 \\
    1     & 1     & 1     & 1 \\
    1/8	  & 1     & 1     & 1 \\
\end{array}
\right] \qquad \text{and} \qquad
\hat{\mathbf{A}} = \left[
\begin{array}{K{1.5em} K{1.5em} K{1.5em} K{1.5em}}
    1     & 2     & 1     & 4 \\
    1/2   & 1     & 1     & 2 \\
    1     & 1     & 1     & 1 \\
    1/4   & 1/2   & 1     & 1 \\
\end{array}
\right].
\]
$\hat{\mathbf{A}}$ can be obtained from $\mathbf{A}$ through an $\alpha$-transformation on the triad $(1,2,4)$ by $\alpha = 2$ as $\hat{a}_{1,2} = 2 a_{1,2}$, $\hat{a}_{1,4} = a_{1,4} / 2$, and $\hat{a}_{2,4} = 2 a_{2,4}$. The corresponding weight vectors are
\[
\begin{array}{ccccc}
\mathbf{w}^{EM}(\mathbf{A}) & \approx & \left[
\begin{array}{cccc}
    0.4269 & 0.2182 & 0.2182 & 0.1367 \\
\end{array}
\right]^\top & \neq & \\
& \neq & \left[
\begin{array}{cccc}
    0.3941 & 0.2256 & 0.2389 & 0.1413 \\
\end{array}
\right]^\top & \approx & \mathbf{w}^{EM}(\hat{\mathbf{A}}),
\end{array}
\]
which shows the violation of the axiom $IT$.
\end{proof}

\begin{proposition} \label{Prop33}
The Logarithmic Least Squares Method satisfies invariance to $\alpha$-transformation on a triad.
\end{proposition}

\begin{proof}
Take two pairwise comparison matrices $\mathbf{A}, \hat{\mathbf{A}} \in \mathcal{A}^{n \times n}$ such that $\hat{\mathbf{A}}$ is obtained from $\mathbf{A}$ through an $\alpha$-transformation on a triad, namely, they are identical except for $\hat{a}_{i,j} = \alpha a_{i,j}$ ($\hat{a}_{j,i} = a_{j,i} / \alpha$), $\hat{a}_{j,k} = \alpha a_{j,k}$ ($\hat{a}_{k,j} = a_{k,j} / \alpha$), and $\hat{a}_{k,i} = \alpha a_{k,i}$ ($\hat{a}_{i,k} = a_{i,k} / \alpha$).
The product of row elements does not change, so $\mathbf{w}^{LLSM}(\mathbf{A}) = \mathbf{w}^{LLSM} ( \hat{\mathbf{A}} )$ according to \eqref{Form_LLSM}.
\end{proof}

\begin{corollary} \label{Cor31}
The counterexample concerning the Eigenvector Method and $IT$ in Proposition~\ref{Prop32} is minimal with respect to the number of alternatives as in the case of $n=3$, $EM$ and $LLSM$ yield the same result \citep{CrawfordWilliams1985}.
\end{corollary}

\section{Characterization of the Logarithmic Least Squares Method} \label{Sec4}

\begin{theorem} \label{Theo41}
The Logarithmic Least Squares Method is the unique weighting method satisfying correctness and invariance to $\alpha$-transformation on a triad.
\end{theorem}

\begin{proof}
$LLSM$ satisfies both axioms according to Propositions~\ref{Prop31} and \ref{Prop33}.

For uniqueness, consider an arbitrary pairwise comparison matrix $\mathbf{A} \in \mathcal{A}^{n \times n}$ and a weighting method $f: \mathcal{A}^{n \times n} \to \mathbb{R}^n$, which meets correctness and invariance to $\alpha$-transformation on a triad.
Denote by $P_i = \sqrt[n]{\prod_{k=1}^n a_{i,k}}$ the geometric mean of row elements for alternative $i$. In order to prove that $f$ is equivalent to the Logarithmic Least Squares Method, it is enough to show that $f_i(\mathbf{A}) / f_j(\mathbf{A}) = P_i / P_j$.

The proof will present that every pairwise comparison matrix can be transformed into a consistent matrix by a sequence of $\alpha$-transformations on a triad carried out over an appropriately chosen set of triads: the triads to be modified always contain the first alternative, while the two others are considered according to the sequence $(n-1,n)$,$(n-2,n)$,$(n-2,n-1)$,$(n-3,n)$,$(n-3,n-1)$,\dots , $(2,n)$,$(2,n-1)$,\dots ,$(2,3)$.
Since the Logarithmic Least Squares Method satisfies invariance to $\alpha$-transformation on a triad, each matrix of the sequence shares the same geometric mean weight vector.
The procedure will stop after (maximally) $(n-1)(n-2)/2$ steps, which is the number of triads not containing the first alternative.

Let us introduce the pairwise comparison matrix $\mathbf{A}^{(n-1,n)} \in \mathcal{A}^{n \times n}$ such that $a_{1,n-1}^{(n-1,n)}:= \alpha_{n-1,n} a_{n-1,n}$; $a_{1,n}^{(n-1,n)}:= a_{1,n} / \alpha_{n-1,n}$; $a_{n-1,n}^{(n-1,n)}:= \alpha_{n-1,n} a_{1,n}$\footnote{~For the sake of simplicity, only the elements above the diagonal are indicated.}
and $a_{i,j}:= a_{i,j}^{(n-1,n)}$ for all other elements, where $\alpha_{n-1,n} = P_{n-1} / \left( P_n a_{n-1,n} \right)$.
Since $\mathbf{A}^{(n-1,n)}$ is obtained from $\mathbf{A}$ through an $\alpha$-transformation on a triad, $f(\mathbf{A}) = f \left( \mathbf{A}^{(n-1,n)} \right)$ according to the assumption that $f$ satisfies the axiom $IT$.
If $n=3$, $\mathbf{A}^{(n-1,n)}$ is consistent because $a_{1,2}^{(2,3)} = P_1 / P_2$, $a_{1,3}^{(2,3)} = P_1 / P_3$, and $a_{2,3}^{(2,3)} = P_2 / P_3$, therefore correctness implies $f(\mathbf{A}) = \mathbf{w}^{LLSM}(\mathbf{A})$.

Otherwise, analogous $\alpha$-transformations on a triad can be implemented until we get $\mathbf{A}^{(i,j)} \in \mathcal{A}^{n \times n}$, where $1 < i < j$ and $a_{k,\ell}^{(i,j)} = P_k / P_\ell$ for all $i \leq k < \ell$. The next step of the algorithm depends on the difference of $i$ and $j$.
\begin{itemize}
\item
If $j > i+1$, then introduce the pairwise comparison matrix $\mathbf{A}^{(i,j-1)} \in \mathcal{A}^{n \times n}$ such that $a_{1,i}^{(i,j-1)}:= \alpha_{i,j-1} a_{1,i}^{(i,j)}$; $a_{1,j}^{(i,j-1)}:= a_{1,j}^{(i,j)} / \alpha_{i,j-1}$; $a_{i,j-1}^{(i,j-1)}:= \alpha_{i,j-1} a_{i,j-1}^{(i,j)}$ and $a_{k,\ell}^{(i,j-1)}:= a_{k,\ell}^{(i,j)}$ for all other elements, where $\alpha_{i,j-1} = P_{i} / \left( P_{j-1} a_{i,j-1}^{(i,j)} \right)$.
It can be checked that $a_{k,\ell}^{(i,j-1)} = a_{k,\ell}^{(i,j)} = P_k / P_\ell$ for all $i \leq k < \ell$, while $a_{i,j-1}^{(i,j-1)} = P_{i} / P_{j-1}$.
\item
If $j = i+1$ and $i > 2$, then define the pairwise comparison matrix $\mathbf{A}^{(i-1,n)} \in \mathcal{A}^{n \times n}$ such that $a_{1,i-1}^{(i-1,n)}:= \alpha_{i-1,n} a_{1,i-1}^{(i,n)}$; $a_{1,n}^{(i-1,n)}:= a_{1,n}^{(i,n)} / \alpha_{i-1,n}$; $a_{i-1,n}^{(i-1,n)}:= \alpha_{i-1,n} a_{i-1,n}^{(i,n)}$ and $a_{k,\ell}^{(i,j-1)}:= a_{k,\ell}^{(i,j)}$ for all other elements, where $\alpha_{i-1,n} = P_{i-1} / \left( P_n a_{i-1,n}^{(i,n)} \right)$.
It can be checked that $a_{k,\ell}^{(i-1,n)} = a_{k,\ell}^{(i,n)} = P_k / P_\ell$ for all $i \leq k < \ell$, while $a_{i-1,n}^{(i-1,n)} = P_{i-1} / P_n$.
\end{itemize}
Finally, $\mathbf{A}^{(2,3)} \in \mathcal{A}^{n \times n}$ is obtained such that $a_{k,\ell}^{(2,3)} = P_k / P_\ell$ for all $2 \leq k < \ell$. Furthermore,
\[
a_{1,j}^{(2,3)} = a_{1,j} \frac{\prod_{m=j+1}^{n} \alpha_{j,m}}{\prod_{m=2}^{j-1} \alpha_{m,j}} = a_{1,j} \left( \prod_{m=j+1}^{n} \frac{P_j}{P_m} \frac{1}{a_{j,m}} \right) \left( \prod_{m=2}^{j-1} \frac{P_j}{P_m} a_{m,j} \right).\footnote{~Note that $\prod_{m=2}^{j-1} \alpha_{m,j} = 1$ if $j=2$ and $\prod_{m=j+1}^{n} \alpha_{m,j} = 1$ if $j=n$.}
\]
However, $a_{m,j} = 1 / a_{j,m}$ due to the reciprocity condition and $\prod_{m=1}^n a_{j,m} = P_j^n$, therefore
\[
a_{1,j}^{(2,3)} = \frac{P_j^{n-2}}{\prod_{m=j+1}^{n} P_m \prod_{m=2}^{j-1} P_m} \frac{1}{P_j^n} = \frac{1}{P_j} \frac{1}{\prod_{m=2}^{n} P_m}.
\]
It is clear that $P_1 = 1 / \left( \prod_{m=2}^{n} P_m \right)$ as the product of all elements of $\mathbf{A}$ gives one, which leads to
\[
a_{1,j}^{(2,3)} = \frac{P_1}{P_j}
\]
for all $j \geq 2$. In other words, $\mathbf{A}^{(2,3)} \in \mathcal{A}^{n \times n}$ is a consistent pairwise comparison matrix such that $a_{i,j}^{(2,3)} = P_i / P_j = w_i^{LLSM} \left( \mathbf{A}^{(2,3)} \right) / w_j^{LLSM} \left( \mathbf{A}^{(2,3)} \right)$ for all $1 \leq i,j \leq n$.
Consequently, $f \left( \mathbf{A}^{(2,3)} \right) = \mathbf{w}^{LLSM} \left( \mathbf{A}^{(2,3)} \right)$ due to correctness.
Weighting method $f$ is invariant to $\alpha$-transformation on a triad, hence
\[
f \left( \mathbf{A}^{(2,3)} \right) = f \left( \mathbf{A}^{(2,4)} \right) = \dots = f \left( \mathbf{A}^{(n-1,n)} \right) = f \left( \mathbf{A} \right).
\]
The Logarithmic Least Squares Method also satisfies $IT$ according to Proposition~\ref{Prop33}, verifying the claim that $f \left( \mathbf{A} \right) = \mathbf{w}^{LLSM} \left( \mathbf{A} \right)$.
\end{proof}

\begin{example} \label{Examp41}
As an illustration of the proof of Theorem~\ref{Theo41}, consider the following pairwise comparison matrix:
\[
\mathbf{A} = \left[
\begin{array}{K{2em} K{2em} K{2em} K{2em}}
    1     & 1     & 1     & 16 \\
    1     & 1     & 1     & 1 \\
    1     & 1     & 1     & 1 \\
    1/16  & 1     & 1     & 1 \\
\end{array}
\right] \text{, which leads to }
\mathbf{w}^{LLSM}(\mathbf{A}) = \frac{1}{9} \left[
\begin{array}{c}
    4 \\
    2 \\
    2 \\
    1 \\
\end{array}
\right].
\]
Since $a_{3,4} \neq w^{LLSM}_3(\mathbf{A}) / w^{LLSM}_4(\mathbf{A})$, an $\alpha$-transformation on the triad $(1,3,4)$ should be carried out by $\alpha_{3,4} = \left[ w^{LLSM}_3(\mathbf{A}) / w^{LLSM}_4(\mathbf{A}) \right] / a_{3,4} = 2$, which results in the matrix $\mathbf{A}^{(3,4)}$.
After that, another $\alpha$-transformation on the triad $(1,2,4)$ is necessary by $\alpha_{2,4} = \left[ w^{LLSM}_3(\mathbf{A}) / w^{LLSM}_4(\mathbf{A}) \right] / a_{2,4}^{(3,4)} = 2$ in order to get the matrix $\mathbf{A}^{(2,4)}$:
\[
\mathbf{A}^{(3,4)} = \left[
\begin{array}{K{1.5em} K{1.5em} K{1.5em} K{1.5em}}
    1     & 1     & 2     & 8 \\
    1     & 1     & 1     & 1 \\
    1/2   & 1     & 1     & 2 \\
    1/8	  & 1     & 1/2   & 1 \\
\end{array}
\right] \qquad \text{and} \qquad
\mathbf{A}^{(2,4)} = \left[
\begin{array}{K{1.5em} K{1.5em} K{1.5em} K{1.5em}}
    1     & 2     & 2     & 4 \\
    1/2   & 1     & 1     & 2 \\
    1/2   & 1     & 1     & 2 \\
    1/4	  & 1/2   & 1/2   & 1 \\
\end{array}
\right] = \mathbf{A}^{(2,3)}.
\]
Finally, an $\alpha$-transformation on the triad $(1,2,3)$ should be carried out by $\alpha_{2,3} = \left[ w^{LLSM}_2(\mathbf{A}) / w^{LLSM}_3(\mathbf{A}) \right] / a_{2,3}^{(3,4)} = 1$, so the pairwise comparison matrix remains unchanged, $\mathbf{A}^{(2,3)} = \mathbf{A}^{(2,4)}$. It is a consistent matrix, therefore any weighting method satisfying correctness and invariance to $\alpha$-transformation on a triad  should give $\mathbf{w}^{LLSM}(\mathbf{A})$ as the weight vector associated with the pairwise comparison matrix $\mathbf{A}$.
\end{example}

\begin{proposition} \label{Prop61}
$CO$ and $IT$ are logically independent axioms.
\end{proposition}

\begin{proof}
It is shown that there exist weighting methods, which satisfy one axiom, but do not meet the other:
\begin{enumerate}[label=\fbox{\arabic*}]
\item
$CO$: the Eigenvector Method (see Propositions~\ref{Prop31} and \ref{Prop32});
\item
$IT$: the flat method such that $f_i(\mathbf{A}) = 1/n$ for all $1 \leq i \leq n$.
\end{enumerate}
\end{proof}

\section{Conclusions} \label{Sec5}

We have proved $LLSM$ to be the \emph{unique} weighting method among the procedures used to derive priorities from reciprocal pairwise comparison matrices, which is correct in the consistent case and invariant to a specific transformation on a triad. The somewhat surprising fact is that our algorithm aims only to recover local consistency by focusing on a given triad without the consideration of other elements of the pairwise comparison matrix. Hence satisfaction of a local property fully determines a global weight vector.

Naturally, one can debate whether the axiom invariance to $\alpha$-transformation on a triad should be accepted, but, at least, it reveals an important aspect of the geometric mean, contributing to the long list of its favourable theoretical properties \citep{BarzilaiCookGolany1987, Barzilai1997, Dijkstra2013, CaklovicKurdija2017, LundySirajGreco2017, Csato2018d}.
Furthermore, the violation of this property can be an argument against the Eigenvector Method, a procedure having several other disadvantages, for example, the Pareto inefficiency of  the weight vector \citep{BlanqueroCarrizosaConde2006, Bozoki2014, BozokiFulop2018}, or the possibility of strong rank reversal in group decision-making \citep{PerezMokotoff2016, Csato2017b}.

Some directions for future research are also worth mentioning.
First, further axiomatic analysis and characterizations of weighting methods can help in a better understanding of them.
Second, $\alpha$-transformation on a triad seems to be related to inconsistency reduction processes in pairwise comparison matrices \citep{KoczkodajSzybowski2016, Szybowski2018}.
Third, the Logarithmic Least Squares Methods has been extended to the incomplete case when certain elements of the pairwise comparison matrix are unknown \citep{BozokiFulopRonyai2010}. Axiomatization on this more general domain seems to be promising and within reach, as revealed by \citet{BozokiTsyganok2017}, although $LLSM$ sometimes behaves strangely on this general domain \citep{CsatoRonyai2016}.

\section*{Acknowledgements}
\addcontentsline{toc}{section}{Acknowledgements}
\noindent
We would like to thank \emph{Denis Bouyssou} for inspiration and \emph{S\'andor Boz\'oki} for plentiful advice. \\
Three anonymous reviewers provided beneficial remarks and suggestions. \\
We are grateful to the audience of our talk at Corvinus Game Theory Seminar on 17 March 2017 for useful comments concerning the proof of Theorem~\ref{Theo41}. \\
The research was supported by OTKA grant K 111797 and by the MTA Premium Post Doctorate Research Program.

\bibliographystyle{apalike}
\bibliography{All_references}

\end{document}